\documentclass[11pt,final]{amsart}
\usepackage{amsmath, amsxtra, amssymb, mathdots}
\usepackage{verbatim}

\usepackage[margin=1.3in]{geometry}
\usepackage{graphicx}
\usepackage[cmtip,all]{xy}

\usepackage[notcite, notref]{showkeys}
\usepackage[colorlinks, linkcolor=blue, citecolor=red, urlcolor=black]{hyperref}

\theoremstyle{plain}
\newtheorem*{main-theorem}{Main Theorem}
\newtheorem{theorem}[equation]{Theorem}
\newtheorem{prop}[equation]{Proposition}
\newtheorem{corollary}[equation]{Corollary}

\newtheorem{conjecture}[equation]{Conjecture}

\newtheorem*{claim*}{Claim}

\newtheorem{lemma}[equation]{Lemma}

\theoremstyle{definition}
\newtheorem{definition}[equation]{Definition}

\newtheorem{assumption}[equation]{Assumption}

\newtheorem{question}[equation]{Question}

\theoremstyle{remark}
\newtheorem{remark}[equation]{Remark}

\numberwithin{equation}{section}

\DeclareMathOperator{\spec}{Spec}
\DeclareMathOperator{\Pic}{Pic}
\DeclareMathOperator{\Sym}{Sym}
\DeclareMathOperator{\proj}{Proj}
\DeclareMathOperator{\Stab}{Stab}

\DeclareMathOperator{\rank}{rank}

\DeclareMathOperator{\Cliff}{Cliff}

\DeclareMathOperator{\SL}{SL}

\DeclareMathOperator{\PGL}{PGL}

\DeclareMathOperator{\Grass}{Grass}
\DeclareMathOperator{\Gr}{Gr}

\def\VV{\mathbb{V}}
\def\GG{\mathbb{G}}

\newcommand{\gitq}{/\hspace{-0.25pc}/}
\newcommand\Mg[1]{\overline{\mathcal{M}}_{#1}}
\newcommand\M{\overline{M}}

\def\noqed{\renewcommand{\qedsymbol}{}}

\def\cS{\mathcal{S}}

\def\cX{\mathcal{X}}
\def\cL{\mathcal{L}}
\def\cF{\mathcal{F}}

\def\cQ{\mathcal{Q}}
\def\cM{\mathcal{M}}
\def\cI{\mathcal{I}}

\def\cU{\mathcal{U}}

\def\cE{\mathbb{E}}

\def\co{\colon\thinspace} 
\def\QQ{\mathbb{Q}}

\def\LL{\mathbb{L}}

\def\PP{\mathbb{P}}

\def\ZZ{\mathbb{Z}}
\def\CC{\mathbb{C}}
\def\HH{\mathrm{H}}

\def\cO{\mathcal{O}}

\DeclareMathOperator{\Cone}{Cone}

\DeclareMathOperator{\Sp}{Syz}
\DeclareMathOperator{\Aut}{Aut}

\newcommand{\s}{\operatorname{ss}}
\newcommand{\cG}{\mathcal{G}}

\makeatletter
\def\blfootnote{\xdef\@thefnmark{}\@footnotetext}
\makeatother

\address[Fedorchuk]{Department of Mathematics\\
Boston College\\
140 Commonwealth Ave\\
Chestnut Hill, MA 02467, USA}
\email{maksym.fedorchuk@bc.edu}

\begin{document}

\title[GIT of syzygies with applications]{Geometric invariant theory of syzygies, with applications to moduli spaces}

\author{Maksym Fedorchuk}

\begin{abstract} 
We define syzygy points of projective schemes, and introduce a program of studying their GIT
stability. Then we describe two cases where we have managed to make some progress in this program,
that of polarized K3 surfaces of odd genus, and of genus six canonical curves.  Applications of our results
include effectivity statements for divisor classes on the moduli space of odd genus K3 surfaces,
and a new construction in the Hassett-Keel program for the moduli space of genus six curves. 

\end{abstract}

\maketitle

\setcounter{tocdepth}{1}
\tableofcontents

\section{Introduction}

Geometric Invariant Theory (GIT) was developed by Mumford \cite{GIT} in order 
to construct algebraic varieties parameterizing orbits
of a reductive linear algebraic group acting on a scheme,
and has since become an important tool 
in the algebraic geometer's arsenal, wielded to great effect in construction of 
various moduli spaces.  A GIT method of constructing compact moduli spaces of polarized varieties was laid out in 
Gieseker's pioneering work \cite{gieseker-surfaces}, and has been applied since then with 
minor variations.  

Gieseker's approach is rooted in the notion of \emph{Hilbert stability}, which we
now recall.  Suppose we are given a sufficiently nice class of polarized varieties with fixed numerical invariants,
\[
\mathfrak{C}=\bigl\{(X,\cL): \text{$X$ is projective with an ample line bundle $\cL$}\bigr\},
\] 
where the objects whose moduli space we desire to construct live. 
The first step is a choice of 
an integer $k$ so that $\cL^k$ defines an embedding of $X$ into a fixed projective space $\PP^{r-1}$
for every $(X,\cL)\in \mathfrak{C}$. 
The ideal sheaf $\cI_X \subset \cO_{\PP^{r-1}}$
then gives a point in the Hilbert scheme $\operatorname{Hilb}(\PP^{r-1})$.
More concretely, fix $q \gg 0$, and let $m = h^0\bigl(X, \cO_X(q)\bigr)$.
Using the short exact sequence
\[
0 \to \HH^0\bigl(\PP^{r-1},\cI_X(q)\bigr) \to \HH^0\bigl(\PP^{r-1},\cO(q)\bigr)\to \HH^0\bigl(X,\cO_X(q)\bigr) \to 0,
\]
Gieseker defines \emph{the $q^{th}$ Hilbert point of 
$X\hookrightarrow \PP^{r-1}$} to be the point of the
Grassmannian $\Gr\left(\HH^0\bigr(\PP^{r-1}, \cO(q)\bigl), m\right)$ represented by the surjection
$ \HH^0\bigl(\PP^{r-1},\cO(q)\bigr)\to \HH^0\bigl(X,\cO_X(q)\bigr)$ (see \S\ref{S:notation} for our conventions concerning Grassmannians).
For $q\gg 0$, the $q^{th}$ Hilbert point determines $(X,\cL)$ up to the action of $\PGL(r)$,
and the assignment $\cI_X \mapsto  \HH^0\bigl(\PP^{r-1}, \cI_X(q)\bigr)$ is a closed embedding 
of $\operatorname{Hilb}(\PP^{r-1})$ into the Grassmannian 
\[
\Gr\left(\HH^0\bigr(\PP^{r-1}, \cO(q)\bigl), m\right) =\Grass\left(h^0\bigl(\PP^{r-1}, \cI_X(q)\bigr), \HH^0\bigl(\PP^{r-1},\cO(q)\bigr)\right).
\]
The pullback of the Pl\"ucker line bundle from this Grassmannian to $\operatorname{Hilb}(\PP^{r-1})$
is an ample $\SL(r)$-linearization, so the general machinery of GIT
produces a quotient parameterizing varieties with GIT {semistable} $q^{th}$ Hilbert points.

If the $q^{th}$ Hilbert point of $X$ is GIT semistable for $q\gg 0$, one says that $X$ is  
\emph{Hilbert semistable}. Many examples of successful GIT applications in moduli theory are obtained 
as GIT quotients of parameter spaces of Hilbert semistable objects in a suitably chosen $\mathfrak{C}$. 
These include GIT constructions of the moduli space $\M_g$ 
of Deligne-Mumford stable curves, the moduli space of surfaces of general type (both by Gieseker,
in \cite{gieseker-cime} and \cite{gieseker-surfaces}, respectively), 
and the moduli space of canonically polarized manifolds (Viehweg \cite{viehweg}).

In this paper, we propose a modification of the Gieseker's method based on
Koszul cohomology. Its basic idea is very simple. Since the $q^{th}$ Hilbert point of $X$ 
is the vector space of degree $q$ hypersurfaces containing $X$, it describes the 
degree $q$ generators of the homogeneous ideal $I_X$ of $X$. It is thus natural to look at higher syzygy 
modules of $I_X$, describing (higher) relations among the generators. Doing
this gives rise to a notion of syzygy points, which depend on two parameters, $p$ and $q$,
and which generalize $q^{th}$ Hilbert points. We formally introduce 
$(p,q)$-syzygy points in Section \ref{S:syzygy}.

Once a syzygy point is defined, GIT enters the picture in two different but related ways. 
The most classical scenario would be to first consider a component $T$ of the Hilbert scheme  
$\operatorname{Hilb}(\PP^{r-1})$
of those polarized schemes in class $\mathfrak{C}$ that live in a fixed projective space $\PP^{r-1}$, just as above. 
But instead of looking at Hilbert semistable objects in $T$, we now consider 
the locus of $t\in T$ such that $X_t$ has a semistable $(p,q)$-syzygy point. If this locus is non-empty,
we consider its GIT quotient by the action of $\SL(r)$ and ask what functor 
does the resulting GIT quotient represent. This question has not been 
systematically studied before for any value of $p$ other than $p=0$, which is the case of Hilbert stability. 
In Section \ref{S:genus-6}, we
give a non-trivial instance of a reasonably complete answer to this question for genus six curves.

The other question that we can ask, in the spirit of a well-known work of Cornalba and Harris \cite{CH},
is what does GIT semistability of a $(p,q)$-syzygy point for a \emph{single} $(X,\cL)\in \mathfrak{C}$ say about the birational geometry of the moduli space for objects in $\mathfrak{C}$? We first recast the Cornalba-Harris
approach in the language of syzygy points in Section \ref{S:families}.
Two examples of $\mathfrak{C}$ to which we apply it
in this paper are polarized K3 surfaces of odd genus and 
Deligne-Mumford stable curves; considered
in Sections \ref{S:k3} and \ref{S:canonical}, respectively.

A disclaimer is in order: 
We do not prove anything new about Koszul cohomology of curves and K3 surfaces.
For an overview of very recent progress in this area the reader is encouraged to consult 
\cite{farkas-survey} and \cite{kemeny-survey}. Rather, 
the aim of this paper is to illustrate that a study of \emph{GIT stability of syzygy points}, 
even in such well-mined cases as curves and K3 surfaces, 
might turn out to be just as fruitful as that of Hilbert points.

\subsection{Notation and conventions}\label{S:notation} We work over $\CC$. A flat family $\pi \co \cX\to T$ of schemes
over an irreducible one-dimensional base is called an \emph{isotrivial specialization} if 
for some $0\in T$, we have that $\cX\times_{T} (T\setminus \{0\})$ is an isotrivial family over $T\setminus \{0\}$.
Given an isotrivial specialization, we say that $X_{t}$ (where $t\neq 0$) isotrivially specializes to $X_0$.

If $V$ is a vector space, we define by $\mathbb S^{\mathbf{\lambda}} (V)$ its Schur functor given 
by a partition $\lambda$. We denote by $\Grass(k, V)$ the Grassmannian of $k$-dimensional subspaces of $V$
and by $\Gr(V, k)$ the Grassmannian of $k$-dimensional quotients of $V$. In Section \ref{S:genus-6},
we use $\epsilon$ to denote a sufficiently small positive rational number.

We assume familiarity with GIT, VGIT, and Koszul cohomology of curves and K3 surfaces, 
as discussed, for example,
in \cite{GIT}, \cite{dolgachev-hu}, and \cite{aprodu-farkas-survey}, respectively.

\subsection*{Acknowledgements} Foremost, this paper owes its existence to the organizers of the Abel Symposium 2017
``Geometry of Moduli,'' who gave me an opportunity 
and motivation to write up this work. I am also indebted to Gavril Farkas, whose influence
is evident in every section of this paper, and who generously 
shared his and Se\'{a}n Keel's ideas to use syzygies as 
the means to construct the canonical model of $\M_g$ 
at an AIM workshop in December 2012.  All results in this paper grew out of my attempt to implement these ideas. 
I am grateful to Anand Deopurkar for his comments and suggestions 
on an earlier version of this paper. During the preparation of this paper,
I was partially supported by the NSA Young Investigator grant H98230-16-1-0061 and Alfred P. Sloan Research Fellowship.

\section{Syzygy points}\label{S:syzygy}
A precise machinery for
defining syzygy points was developed by Green under the name of Koszul cohomology \cite{green}. 
We briefly sketch it in the way most amenable to the application of GIT in what follows. 

Let $V$ be a $\CC$-vector space of dimension $r$ and $S=\Sym V$ be its symmetric algebra. Then the minimal graded free resolution of
$\CC$ as an $S$-module is given by the {Koszul complex} $\mathbb{K}(S)$:
\begin{equation*}\label{E:koszul-complex}
0\to \bigwedge^r V \otimes S(-r) \to  \bigwedge^{r-1} V \otimes S(-r+1) \to \cdots \to \bigwedge^2 V \otimes S(-2)
\to V\otimes S(-1) \to S. 
\end{equation*}
In homological degree $p$, and graded degree $p+q$, the differential 
$d_{p,q} \colon \wedge^{p} V \otimes \Sym^{q} V \to \wedge^{p-1} V \otimes \Sym^{q+1} V$ is given by 
\begin{equation*}\label{E:koszul-differential}
d_{p,q}(x_0\wedge x_1\wedge \cdots\wedge x_{p-1}\otimes y)
=\sum_{i=0}^{p-1} (-1)^i \, x_0\wedge \cdots \wedge \widehat{x_i}\wedge \cdots \wedge x_{p-1}\otimes x_i y.
\end{equation*}

Now given an arbitrary graded $S$-module $R$, the Koszul complex of $R$ is 
defined simply as $\mathbb{K}(R):=\mathbb{K}(S)\otimes R$. If $d_{p,q}^{R}$ 
are the Koszul differentials in $\mathbb{K}(R)$, then the {Koszul cohomology groups of $R$} are defined to be
\begin{equation*}\label{E:koszul-cohomology}
K_{p,q}(R):=\HH_{p}(\mathbb{K}(R))_{p+q} =\ker d^R_{p,q}/\operatorname{im} d^R_{p+1,q-1} \simeq \operatorname{Tor}_{p}^{S}(R, \CC)_{p+q}. 
\end{equation*}

If $\varphi\colon S\to R$  
is a ring homomorphism giving $R$ the structure of a graded $S$-algebra, and $I=\ker(\varphi)$, 
then comparing the Koszul complexes of $S$, $R$, and $I$, 
we obtain in graded degree $p+q$, and homological degrees $p+1,p, p-1$, 
the following commutative diagram of vector spaces
\begin{equation}\label{D:koszul}
\begin{gathered}
\small{
\xymatrix{
\bigwedge^{p+1}V \otimes I_{q-1}  \ar[rr]^-{d^{I}_{p+1,q-1}}\ar@{^{(}->}[d]  & 
&  \bigwedge^{p}V \otimes I_q \ar[rr]^-{d^{I}_{p,q}}\ar@{^{(}->}[d] 
& & \bigwedge^{p-1}V \otimes I_{q+1}\ar@{^{(}->}[d] \\
\bigwedge^{p+1}V \otimes \Sym^{q-1} V  \ar[rr]^-{d_{p+1,q-1}}\ar@{->}[d]  & 
&  \bigwedge^{p}V \otimes \Sym^{q} V \ar[rr]^-{d_{p,q}}\ar@{->}[d] 
& & \bigwedge^{p-1}V \otimes \Sym^{q+1} V\ar@{->}[d] \\
\bigwedge^{p+1}V \otimes R_{q-1} \quad \ar[rr]^-{d^R_{p+1,q-1}} & & \bigwedge^{p}V \otimes R_{q} \ar[rr]^-{d^R_{p,q}}
& &\bigwedge^{p-1}V \otimes R_{q+1}
}}
\end{gathered}
\end{equation}
Suppose now that the following is satisfied:
\begin{assumption}\label{assumption-koszul}
The map $\varphi_{q-1}\colon \Sym^{q-1} V \to R_{q-1}$ is surjective, and $K_{p,q}(R)=0$.
\end{assumption}
Then from Diagram \eqref{D:koszul},  
we obtain a surjection $\ker(d_{p,q}) \twoheadrightarrow \ker (d^R_{p,q})$. 
The decomposition of the Koszul complex of $S$ into $\SL(V)$-representations
is well-known. In particular, $\ker(d_{p,q})$ is an irreducible 
$\SL(V)$-representation $\mathbb S^{\mathbf{\lambda}_{p,q}}(V)$ given by the partition $\mathbf{\lambda}_{p,q}:=(q,\underbrace{1,\dots,1}_{p})$ (see \cite[p.83, Exercise 6.20(c)]{fulton-harris}). With this in mind, we make the following definition:
\begin{definition}
Under Assumption \ref{assumption-koszul}, 
we define the \emph{$(p,q)$-syzygy point}
of $R$ to be the point of the Grassmannian 
\[\Gr\left(\mathbb S^{\mathbf{\lambda}_{p,q}}(V), \dim \ker (d^R_{p,q})\right)\] 
given by the exact sequence $
\mathbb S^{\mathbf{\lambda}_{p,q}}(V) \to 
 \ker d^R_{p,q} \to 0.
$
\end{definition}
By construction, there is a natural $\SL(V)$-action on the set of $(p,q)$-syzygy points of all $S$-algebras $R$ with a fixed value of $\dim \ker (d^R_{p,q})$, 
induced by the $\SL(V)$-action on $\mathbb S^{\mathbf{\lambda}_{p,q}}(V)$.
We now transition to geometry. 

\begin{definition}
\label{D:syzygy} Suppose $X$ is a projective scheme
and $\cL$ is a line bundle on $X$. Set $V=\HH^0(X,\cL)$, $S=\Sym V$, and
let $R(X,\cL)=\oplus_{n\geq 0} \HH^0(X,\cL^n)$ be the graded section ring of $\cL$. 
If Assumption \ref{assumption-koszul} is satisfied for the $S$-algebra $R(X,\cL)$, then the $(p,q)$-syzygy point of
$R(X,\cL)$ will be called the $(p,q)$-syzygy point of $(X,\cL)$, and denoted
$\Sp_{(p,q)}(X,\cL)$, or   $\Sp_{(p,q)}(X)$ if the line bundle $\cL$ is understood. 
\end{definition}
We will be almost exclusively concerned with 
the situation where $\cL$ is a very ample line bundle embedding $X$ 
as a projectively normal subscheme of a projective space $\PP V^{\vee}$
with a homogeneous ideal $I_X \subset S$. 
In this case, the syzygy point $\Sp_{(p,q)}(X)$ is well-defined if and only if $K_{p,q}(X)=0$.  

While there is no a priori reason not to consider all allowable values of $(p,q)$,
we will see that one obtains a lot of mileage out of just two special cases. 
The first one is $p=0$. In this case, we see from Definition \ref{D:syzygy} that the $(0,q)$-syzygy point
of $X$ is the short exact sequence
\[
0 \to (I_X)_{q} \to \Sym^q V \to \HH^0(X, \cL^q) \to 0,
\]
which is exactly the $q^{th}$ Hilbert point of $X\hookrightarrow \PP V^{\vee}$. 
The second case is $q=2$ and arbitrary $p$, which corresponds to 
$p^{th}$-order linear syzygies among the quadrics cutting out $X$. 
Note that in this case
\[
\ker(d_{p,2}^{I_X})=\ker\left( \wedge^{p} V\otimes (I_X)_2  \to \wedge^{p-1} V \otimes (I_X)_3\right)=K_{p,2}(I_X) \simeq K_{p+1,1}(X).
\]
It follows that $\Sp_{(p,2)}(X)$ is a point of $\Grass\left(\dim K_{p+1,1}(X), \ \mathbb S^{\lambda_{p,2}} (V)\right)$,
where \[\mathbb S^{\lambda_{p,2}} (V)= \dfrac{\wedge^{p+1} V \otimes V}{\wedge^{p+2} V}.\]

Having defined syzygy points, we will be guided by a series of the following open ended questions:
\begin{question}\label{Q:syzygy}
When is the $(p,q)$-syzygy point of $X \subset \PP^{r-1}$ semistable with respect to the $\SL(r)$-action?
Are there geometric manifestations of this stability? Is there a reasonable moduli space of schemes with semistable $(p,q)$-syzygy points?  
\end{question}
While GIT and Koszul cohomology date 
back more than $30$ years ago, the problem of GIT stability of syzygy points of projectively normal varieties,
already alluded to by Green \cite[Problem (5.21)]{green}, has not attracted much attention. A possible reason 
for this is the enormous difficulty in verifying stability of syzygy points even for such basic objects as canonical
curves.    
Aside from powerful results of Kempf \cite{kempf} and Luna \cite{luna} 
that allow to bypass the numerical criterion entirely, and 
often to immediately obtain semistability of \emph{all} syzygy points at once, 
there is virtually nothing else that is known about stability of syzygy points when $p\geq 1$, or even when $p=0$ 
and $q$ is small. We summarize these results as follows:

\begin{theorem}[Kempf-Luna semistability criterion]\label{T:kempf-luna}
Let $X\hookrightarrow \PP V^{\vee}$ be a linearly normal subscheme with 
a well-defined $(p,q)$-syzygy point. 
Let $\Stab(X) \subset \SL(V)$ be the stabilizer of $X$. Suppose $G\subset \operatorname{Stab}(X)$ is a reductive subgroup. 
\begin{enumerate}
\item[(a)]
Suppose $V=\HH^0\bigl(X,\cO(1)\bigr)$ is an irreducible representation of $G$.
Then $\Sp_{(p,q)}(X)$  is semistable with respect to the $\SL(V)$-action. 
\item[(b)] $\Sp_{(p,q)}(X)$ is semistable with respect to the $\SL(V)$-action 
if and only if it is semistable 
with respect to the action of the centralizer $C_{\SL(V)}(G)$.
\end{enumerate}
\end{theorem}
 
\begin{proof} Part (a) follows from \cite[Cor. 5.1]{kempf} and Part (b) from \cite[Cor. 2 and Rem. 1]{luna}.
\end{proof}

\section{Families of polarized varieties}
\label{S:families}
In this section, we recall the Cornalba-Harris method \cite{CH} of producing effective divisors
on families of polarized varieties using GIT stability of a generic fiber.
Suppose $\pi\colon \cX \to T$ is a flat family of projective schemes,
over an integral base $T$, and $\cL$ is a line bundle on $\cX$.
We assume further that: 
\begin{assumption}\label{assumption-family}
The sheaves $R^k\pi_*(\cL^i)$ are locally free for all $i\geq 1$ and $k\geq 0$.
For a generic $t\in T$,  the line bundle $\cL_{t}$ is very 
ample on the fiber $X_t$ and embeds $X_t$ as a projectively normal subscheme of 
$\PP \HH^0(X_t, \cL_t)^{\vee}$. 
\end{assumption}
Our main applications will be to the examples 
where $\pi\colon \cX \to T$ is an exhausting family of a (often Deligne-Mumford)
stack of polarized varieties of a certain type (e.g., canonical curves or K3 surfaces), and so the reader can safely 
specialize to that case in what follows. 


Set $r:=\rank(\pi_*\cL)$. Possibly after passing to a finite cover of $T$, we can assume
that the line bundle $\det(\pi_*\cL)$ is divisible by $r$ in $\Pic(T)$. 
We then choose a line bundle $\cQ\in \Pic(T)$ such that $\cQ^{r}\simeq
\det\bigl(\pi_*\cL\bigr)$ and let $\LL:=\cL\otimes \pi^*(\cQ^{-1})$. 
Then 
\[
\det\bigl(\pi_*\LL \bigr)\simeq \cO_T.
\]
Set now $\cE=\cE_1:=\pi_*\LL$ and $\cE_{i}:=\pi_*\LL^{i}$, for $i\geq 2$. 
By Assumption \ref{assumption-family}, we have that $\cE_{i}$ is a locally free sheaf for every $i\geq 1$. 
We note for the future use that 
\begin{equation*}
c_1(\cE_{i})=c_1(\pi_*\cL^i)-i \rank(\cE_i)c_1(\cQ).
\end{equation*}

Consider next the following commutative diagram of locally free sheaves on $T$:
\begin{equation*}
\xymatrix{
\bigwedge^{p+1}\cE \otimes \Sym^{q-1} \cE \ar[r]  \ar@{->}[d] & 
\bigwedge^{p} \cE \otimes \Sym^{q} \cE \ar[r] \ar@{->}[d] & 
\bigwedge^{p-1} \cE \otimes \Sym^{q+1} \cE \ar@{->}[d]  \ar[r]  & 
\cdots \\
\bigwedge^{p+1} \cE \otimes \cE_{q-1} \ar[r] & 
\bigwedge^{p}\cE \otimes \cE_{q} \ar[r] & 
\bigwedge^{p-1}\cE \otimes \cE_{q+1} \ar[r] & 
\cdots.
}
\end{equation*}
Here, the Koszul differential $\bigwedge^{p}\cE \otimes \Sym^{q} \cE \to  
\bigwedge^{p-1}\cE \otimes \Sym^{q+1} \cE$ (resp., $\bigwedge^{p}\cE \otimes \cE_{q} \to \bigwedge^{p-1}\cE \otimes \cE_{q+1}$) is defined by 
\begin{equation*}
x_0\wedge x_1\wedge \cdots\wedge x_{p-1}\otimes y
\mapsto \sum_{i=0}^{p-1} (-1)^i \, x_0\wedge \cdots \wedge \widehat{x_i}\wedge \cdots \wedge x_{p-1}\otimes x_i y,
\end{equation*}
where $\{x_i\}_{i=0}^{p}$ are local sections of $\cE$, $y$ is a local section 
of $\Sym^{q} \cE$ (resp., $\cE_{q}$), and where $x_iy$ is the local section of $\Sym^{q+1} \cE$
(resp., $\cE_{q+1}$) given by a natural multiplication map.  By cohomology and base change \cite[Corollaire 6.9.9]{EGAIII2}, 
on fibers at $t\in T$, this diagram is nothing but the two bottom rows of Diagram \eqref{D:koszul}
for $R=R(X_t, \cL_t)$.

\begin{definition}\label{D:sheafS} We set
$
\cS_{p,q}:=\ker\left(\bigwedge^{p}\cE \otimes \cE_{q} \to \bigwedge^{p-1}\cE \otimes \cE_{q+1}\right) 
$.  Note that $\cS_{p,q}$ is a coherent sheaf on $T$ such that, by cohomology and base change,
its fiber at $t\in T$ is precisely $\ker(d_{p,q}^{R(X_t, \cL_t)})$. 
\end{definition}

Let $\cU_{p,q} \subset T$ be the open subset of points $t\in T$ such that 
$K_{p-i,q+i}(X_t)=0$ for all $0 \leq i \leq p$, and $\Sym^{q-1} \cE \to \cE_{q-1}$
is surjective.
Then on the open $\cU_{p-1,q+1}\subset T$, we have a long exact sequence of coherent sheaves 
\[
0 \to \cS_{p,q} \to \bigwedge^{p}\cE \otimes \cE_{q} \to \bigwedge^{p-1}\cE \otimes \cE_{q+1} \to \cdots \to  \cE_{p+q} \to 0.
\]
Since $\cE_{i}$'s are locally free by assumption, and $T$ is connected and reduced, 
we conclude that $\cS_{p,q}$ is a locally free sheaf on $\cU_{p-1,q+1}$. 
Moreover,
on the open $\cU_{p,q}\subset \cU_{p-1,q+1}$, we have a surjection 
\begin{equation*}
\mathbb S^{\mathbf{\lambda}_{p,q}}(\cE) \to \cS_{p,q} \to 0,
\end{equation*}
so that $\Sp_{(p,q)}(X_t)$ is well-defined for all $t\in \cU_{p,q}$.


\begin{theorem}\label{T:families}
Suppose that for a generic $t\in \cU_{p,q}$, 
the $(p,q)$-syzygy point of $X_t \subset \PP\HH^0(X_t, \cL_t)^{\vee}$ is semistable with respect to the 
$\SL(r)$-action. Then a positive multiple of the following divisor class
\begin{equation}\label{E:divisor-class}
c_1(\cS_{p,q})=\sum_{i=0}^p (-1)^i \binom{r}{p-i} c_1(\cE_{q+i})
\end{equation}
is effective on $\cU_{p-1,q+1}$.
Moreover, the stable base locus of $c_1(\cS_{p,q})$ on $\cU_{p-1,q+1}$ is contained inside 
\[
\{t \mid \ \text{$\Sp_{(p,q)}(X_t)$ is either undefined, or unstable with respect to the $\SL(r)$-action}\}.
\]
\end{theorem}
\begin{proof} This follows by the original argument of Cornalba and Harris \cite{CH}, with only 
cosmetic modifications.
Suppose $X_t$ has a semistable well-defined $(p,q)$-syzygy point, as given by Definition \ref{D:syzygy}. Set
$V=\HH^0(X_t, \cL_t)$, and let 
$W=\bigwedge^{N} \mathbb S^{\mathbf{\lambda}_{p,q}}(V)$, where $N=\dim \ker(d_{p,q}^{X_t})$.
By definition, semistability of $\Sp_{(p,q)}(X_t)$ is equivalent to the existence
of an $\SL(r)$-invariant polynomial $f\in \Sym^{d} W$
that maps to a non-zero element
of $\Sym^d \left(\bigwedge^{N} \cS_{p,q}\right)=\bigl(\det \cS_{p,q}\bigr)^{d}$.

Since $\det \cE \simeq \cO_T$, the $\SL(r)$-invariance of $f$ implies that $f$ is invariant under
the transition matrices of the vector bundle
$\Sym^{d} \left(\wedge^{N} \mathbb S^{\mathbf{\lambda}_{p,q}}(\cE)\right)$.
We conclude that $f$ gives rise to a global section of 
$\Sym^{d} \left(\wedge^{N} \mathbb S^{\mathbf{\lambda}_{p,q}}(\cE)\right)$ that maps to a section
of $\bigl(\det \cS_{p,q}\bigr)^{d}$ not vanishing at $t\in T$. The claim follows.
\end{proof}

\subsection{Summation formulae}
\label{S:summation}
We give a combinatorial summation formula that will be used in this paper to evaluate expressions
arising from Equation \eqref{E:divisor-class}.

\begin{lemma}\label{L:summation} We have
\begin{equation*}
\sum_{i=0}^{p} (-1)^i \binom{r}{p-i}\binom{i}{a}  = (-1)^a \binom{r-1-a}{p-a}.  
\end{equation*}
\end{lemma}

\begin{proof}
Indeed, the sum on the left multiplied by $(-1)^{p}a!$ is the coefficient of $x^{p-a}$ in
\[
(1-x)^{r}\frac{d^a(\sum_{i=0}^{\infty} x^i)}{dx^a} =(1-x)^{r} \frac{(-1)^a a!}{(1-x)^{a+1}}=(-1)^a a! (1-x)^{r-a-1}.
\]
\end{proof}

\section{K3 surfaces}
\label{S:k3}

Somewhat surprisingly, we can prove GIT stability for more syzygy points of generic K3 surfaces 
than for syzygy points of any other non-homogeneous variety.  
Namely, we have the following:
\begin{theorem}\label{T:k3} Suppose $g=2k+1 \geq 3$.
Let $X\subset \PP^{g}$ be a generic polarized K3 surface of degree $2g-2$.
Then the syzygy points $\Sp_{(p,q)}(X)$ are semistable with respect to the $\SL(g+1)$-action 
for all $(p,q)$ such $K_{p,q}(X)=0$. In particular, the $q^{th}$ Hilbert points $\Sp_{(0,q)}(X)$ of $X$
are semistable for every $q\geq 2$, and the $p^{th}$-order linear syzygies among quadrics
$\Sp_{(p,2)}(X)$ are semistable for every $1\leq p\leq k-1$. 
\end{theorem}
Note that for a generic polarized K3 surface of genus $g$, we have $K_{p,q}(X)=0$ for $q\geq 3$,
except when $(p,q)=(2k-1,3)$, by the self-duality of the Betti table of $X$, and we have $K_{p,2}(X)=0$ for $p\leq k-1$ by Voisin's proof of the generic Green's conjecture \cite{voisin-green-odd}.

Theorem \ref{T:k3} is a generalization (in odd genus only)
of a very recent result of Farkas and Rim\'anyi:
\begin{theorem}[{\cite[Theorem 10.2]{farkas-rimanyi}}]
Let $X$ be a polarized K3 surface of degree $2g-2$ with $\Pic(X)\simeq \ZZ$. 
Then the 2nd Hilbert point $\Sp_{(0,2)}(X)$ is semistable.
\end{theorem}
Farkas and Rim\'anyi obtain their result by considering 
the locus of genus $g$ polarized K3 surfaces whose embedding in $\PP^g$
lies on a rank $4$ quadric. This locus is divisorial on the moduli space, 
and in fact is induced by an $\SL(g+1)$-invariant 
divisor on the Grassmannian parameterizing 2nd Hilbert points of such surfaces. 
It is an interesting problem to find similar geometric divisorial conditions on higher $(p,2)$-syzygy points
of K3 surfaces that would lead to a different proof of their semistability. 

Another instance of Theorem \ref{T:k3} that was previously known is due to Morrison
who proved Hilbert 
stability of every polarized K3 surface $X\subset \PP^g$ with 
$\Pic(X)\simeq \ZZ$  \cite{morrison-K3}. Morrison's result 
covers the case of $p=0$ and $q\gg 0$, but his methods 
are asymptotic in nature and so cannot be used for small values
of $q$ or for positive values of $p$. 

\begin{proof}[Proof of Theorem \ref{T:k3}] 
Define $S_{2k+1}$ to be the closure in $\PP^{2k+1}$ of the image of a non-reduced affine scheme
$\spec \CC[s,t,\varepsilon]/(\varepsilon^2)$
given by the following morphism 
\begin{equation}\label{E:carpet}
\begin{aligned}
x_0&=1, &  x_{k+1}&=s, \\
&\vdots &   &\vdots \\
x_i&=t^i, &  x_{k+i+1}&=st^i+it^{i-1}\varepsilon, \\
\vdots &  & &\vdots \\
x_k&=t^k, &  x_{2k+1}&=st^k+kt^{k-1}\varepsilon. 
\end{aligned}
\end{equation}
One verifies that $S_{2k+1}$ is a projective subscheme of $\PP^{2k+1}$ that 
is covered by four affine charts isomorphic to
$\spec \CC[s,t,\varepsilon]/(\varepsilon^2)$, and has a trivial dualizing line bundle. More precisely, 
$S_{2k+1}$ is a \emph{K3 carpet}, as defined by Bayer and Eisenbud in \cite[Section 8]{BE}.
Moreover, $S_{2k+1}$ is a flat degeneration of smooth $K3$ surfaces of degree $4k$ in $\PP^{2k+1}$,
as shown by  Gallego and Purnaprajna in \cite{k3-carpets}. 
For example, $S_3\subset \PP^3$ is a double quadric given by the equation $(x_0x_4-x_1x_3)^2=0,$
and $S_5$ is a $(2,2,2)$ complete intersection in $\PP^5$ cut out by the quadrics
$$x_0x_2-x_1^2=x_3x_5-x_4^2=x_0x_5+x_2x_3-2x_1x_4=0.$$

Notice that $(S_{2k+1})_{\operatorname{red}}$ is a balanced rational normal scroll embedded in $\PP^{2k+1}$
by the complete linear system $\vert\cO_{\PP^1\times \PP^1}(1,k)\vert$. 
In \cite[Theorem 1.3]{k3-carpets}, Gallego and Purnaprajna also prove that the rational normal scroll supports 
a unique double structure that is numerically $K3$. This uniqueness result implies that the natural action of
$\SL(2) \times \SL(2)$ on $\PP^1\times \PP^1$ extends to the $\SL(2) \times \SL(2)$-action on $S_{2k+1}$,
which can also be seen from our explicit parameterization of $S_{2k+1}$.
It follows by the Borel-Weil theorem that 
$\HH^0\bigl(X,\cO_X(1)\bigr)\simeq \HH^0\bigl(\PP^1\times \PP^1, \cO_{\PP^1\times \PP^1}(1,k)\bigr)$
is an irreducible representation of $\Stab_{\SL(2k+2)}(S_{2k+1})$. 
Applying Theorem \ref{T:kempf-luna},
we conclude that all well-defined syzygy points $\Sp_{(p,q)}(S_{2k+1})$ are semistable. 

A straightforward computation 
shows that $S_{2k+1}$ is projectively normal, which implies that the syzygy point
$\Sp_{(p,q)}(S_{2k+1})$ is well-defined for all $(p,q)$ such that $K_{p,q}(S_{2k+1})=0$. 
It remains to note that $K_{p,q}(S_{2k+1})=0$ if and only if $K_{p,q}(X)=0$ for a generic polarized K3 surface
of genus $2k+1$. This follows from a result of Deopurkar \cite{deopurkar-carpets}, who deduced the vanishing of 
$K_{k-1,2}(S_{2k+1})$ from Voisin's proof of the generic Green's conjecture \cite{voisin-green-odd}. The statement of the theorem 
for a generic K3 surface now follows by openness of semistability.  
\end{proof}

\subsection{Effective divisors on the moduli space of K3 surfaces}
To apply the generic syzygy stability results of Theorem \ref{T:k3} using the framework of
Section \ref{S:families}, we take 
$T=\cF_g$ to be the moduli stack of quasi-polarized K3 surfaces of genus $g$,
parameterizing pairs $(X,L)$, where $X$ is a smooth K3 surface, and $L$
is a big and nef line bundle with $L^2=2g-2$. We then let
$\pi\colon \cX\to T$ be the universal family, and take $\LL$ to be the universal 
polarization on $\cX$ normalized so that $c_1(\pi_* \LL)=0$. 
Having introduced $\pi \colon \cX \to T$ and $\LL$, we 
follow notation of Section \ref{S:families}, and in particular consider the sheaves $\cS_{p,q}$
on $\cF_g$ introduced in Definition \ref{D:sheafS}.

Before proceeding, we need to recall some standard results on tautological classes on $\cF_g$ from \cite{farkas-rimanyi} and \cite{MOP},
whose notation we follow.  First, the tautological sheaves $\cE_i:=\pi_*(\LL^i)$
are locally free, of rank $2+i^2(g-1)$. Second, by \cite[Proposition 9.2]{farkas-rimanyi}, we have
\begin{equation}\label{E:k3-hodge}
c_1(\cE_i)=\frac{i}{12}\kappa_{1,1}+\frac{i^3}{6}\kappa_{3,0}-\left(\frac{(g-1)i^2}{2}+1\right)\lambda,
\end{equation}
where $\kappa_{a,b}=\pi_*\bigl(c_1(\LL)^{a}\cdot c_2(\mathcal T_{\pi})^{b}\bigr)$,
and $\mathcal T_{\pi}$ is the relative tangent sheaf $\pi\colon \cX\to T$.
By our choice of normalization, we have $c_1(\pi_* \LL)=0$. Thus $\kappa_{1,1}=(6g+6)\lambda-2\kappa_{3,0}$,
and so in terms of the canonical tautological divisors $\lambda$ and $\gamma=\kappa_{3,0}-\frac{g-1}{4}\kappa_{1,1}$, we have 
\begin{align}
\kappa_{1,1} &=12\lambda-\frac{4}{g+1}\gamma, \\
\kappa_{3,0}&=3(g-1)\lambda+\frac{2}{g+1}\gamma. \label{E:k30}
\end{align}

We begin with an effectivity result that holds uniformly over the whole moduli space.
\begin{theorem}\label{T:k3-hilbert} Suppose $g\geq 3$ is odd. For every $q\geq 2$, the divisor class
\[
c_1(\cE_q)=\frac{q}{12}\kappa_{1,1}+\frac{q^3}{6}\kappa_{3,0}-\left(\frac{(g-1)q^2}{2}+1\right)\lambda
\]
is effective on $\cF_g$. In particular, $\kappa_{3,0}$ is pseudoeffective. 
\end{theorem}
\begin{proof} Consider the vector bundle $\cS_{0,q}=\cE_q$ on $\cF_g$. 
By Theorem \ref{T:k3}, a generic deformation of the K3 carpet $S_{2k+1}$ 
has a semistable $(0,q)$-syzygy point. Theorem \ref{T:families} now implies that $c_1(\cS_{0,q})$
is an effective $\QQ$-line-bundle on $\cF_g$. Moreover, the closure of the stable base locus 
of $c_1(\cS_{0,q})$ inside the stack of all smoothable
numerically K3 surfaces does not contain $[S_{2k+1}]$. 

By taking $q\to \infty$, we see that $\kappa_{3,0}$ lies in the closure of the effective cone of $\cF_{g}$.
\end{proof}

When $q=2$, we find that $c_1(\cE_2)$ is proportional to 
$[D_{g}^{\operatorname{rk} 4}]$, the class of the Farkas-Rim\'anyi 
divisor of K3 surfaces lying on rank $4$ quadrics; 
see \cite[Theorem 3.1]{farkas-rimanyi} and the surrounding
discussion for the definition and properties of this geometric divisor. 
Theorem \ref{T:k3-hilbert} however produces a different
effective divisor spanning the same ray as $[D_{g}^{\operatorname{rk} 4}]$ in the effective cone of $\cF_g$,
as the following result illustrates.
\begin{corollary}\label{C:k3-hilbert} Suppose $g\geq 5$ is odd. 
The Kodaira-Itaka dimension of the divisor class
\[
c_1(\cE_2)=(2g-1)\lambda+\frac{2}{g+1}\gamma
\]
on $\cF_{g}$ is at least $1$.  
\end{corollary}
\begin{proof}
Since the K3 carpet $S_{2k+1}$ lies on many rank $4$ quadrics as soon as $k\geq 2$, there exists
a smooth deformation $X$ of $S_{2k+1}$ such that $\Sp_{(0,2)}(X)$ is semistable and 
$X$ lies on a rank $4$ quadric. It then follows by Theorem \ref{T:families} that
there is an effective multiple of the divisor class $c_1(\cE_2)$ that avoids $[X]\in \cF_g$. 
The claim follows.
\end{proof}

More generally, the generic semistability result of Theorem \ref{T:k3} shows
that $c_1(\cS_{p,q})$ is an effective divisor class 
on the open locus $U_{p-1,q+1} \subset \cF_{g}$. 
While we will not pursue a detailed discussion of these effectivity results, we at least compute the resulting 
divisor class. By Equations \eqref{E:divisor-class} and \eqref{E:k3-hodge}, we have
\begin{equation*}
c_1(\cS_{p,q}) =\sum_{i=0}^p (-1)^i \binom{g+1}{p-i} \left(\frac{(q+i)}{12}\kappa_{1,1}+\frac{(q+i)^3}{6}\kappa_{3,0}-\left(\frac{(g-1)(q+i)^2}{2}+1\right)\lambda\right).
\end{equation*}
Using summation formulae of Lemma \ref{L:summation}, 
this evaluates to an unwieldy 
\begin{multline}\label{E:K3-slopes}
c_1(\cS_{p,q}) =\frac{1}{12}\left(q\binom{g}{p}-\binom{g-1}{p-1} \right)\kappa_{1,1} \\ 
+\frac{1}{6}\left(q^3\binom{g}{p}-(3q^2+3q+1)\binom{g-1}{p-1}+(6q+6)\binom{g-2}{p-2}-6\binom{g-3}{p-3}\right)\kappa_{3,0} \\
- \frac{g-1}{2}\left(q^2\binom{g}{p}-(2q+1)\binom{g-1}{p-1}+2\binom{g-2}{p-2}\right)\lambda-\binom{g}{p}\lambda.
\end{multline}

For $q=2$, the above formula simplifies significantly yielding
\begin{equation}\label{E:Sp2}
c_1(\cS_{p,2})=\binom{g-2}{p}\left(\left(1 - \frac{p}{g-2}\right)\kappa_{3,0} 
-\left(g-1-\frac{g - 1}{g - p - 1}\right) \lambda\right).
\end{equation}
We remark that $\cS_{p,2}$ coincides with the sheaf $\mathcal{G}_{p,2}$
from \cite[Section 9]{farkas-rimanyi}.
Recall also from \cite[Section 9.2]{farkas-rimanyi} 
that there exists a Noether-Lefschetz divisor $D_{1,1} \subset \cF_{g}$ such that
for every $(X,L)\in \cF_{g}\setminus D_{1,1}$ the line bundle $L$ is base-point-free on $X$.
\begin{corollary}\label{C:Sp2} Suppose $g=2k+1$. Then for every $0\leq p\leq k-1$, the following
divisor class is effective on $\cF_{g}\setminus D_{1,1}$:
\begin{equation}\label{E:Sp2-2}
\frac{(g-1)(2g-3p-1)}{g-p-1}\lambda
+\frac{2}{g+1}\gamma.
\end{equation}
\end{corollary}

\begin{proof}
We only need to observe that by the self-duality of the Betti table, every $(X,L)\in \cF_{g}\setminus D_{1,1}$ satisfies $K_{p-1,3}(X)=\cdots=K_{0,p+2}(X)=0$. It follows that $\cS_{p,2}$ is a locally free sheaf on $\cF_{g}\setminus D_{1,1}$.
The claim now follows from Theorem \ref{T:k3} and Theorem \ref{T:families} by substituting \eqref{E:k30}
into \eqref{E:Sp2}.
\end{proof}

For $p=0$, the divisor class of Corollary \ref{C:Sp2} reduces to that of Corollary \ref{C:k3-hilbert} for $q=2$.
For $p=(g-3)/2$, the divisor class of Corollary \ref{C:Sp2} is proportional to the Koszul divisor $\mathfrak{Kosz}_g$
of Farkas and Rim\'anyi \cite[Theorem 9.5]{farkas-rimanyi}; this of course is immediate from the fact the 
$(k-1,2)$-syzyzy
point of the generic K3 surface $X$ of genus $2k+1$ lives in a $0$-dimensional Grassmannian,
and so the locus where this point is not semistable is precisely $\mathfrak{Kosz}_g$, a degeneracy locus
of a map between two vector bundles of the same rank over $\cF_g$.

\begin{remark}
While it is likely that none of the effective divisors we produce on $\cF_g$ are new, we note that 
our proof of the effectivity of the divisor classes given by Theorem \ref{T:k3-hilbert} does not rely on 
the knowledge of the Betti tables for generic K3 surfaces and so is completely independent of 
the Green's conjecture (Voisin's theorem).
\end{remark}

\section{Canonical curves}
\label{S:canonical}
Having established GIT semistability of syzygy points of generic K3 surfaces (of odd genus), it is natural 
to expect that we can prove an analogous result for their hyperplane sections, namely, canonical curves. Unfortunately, this is outside
of our reach at the moment, so we will be content with stating several conjectures and drawing 
some consequences out of them.  We begin with the following:
\begin{conjecture}\label{conj-can} Suppose $C\subset \PP^{g-1}$ is a generic canonical curve. Then the $(p,2)$-syzygy 
point $\Sp_{(p,2)}(C)$ is GIT semistable for all $p\geq \lfloor (g-3)/2\rfloor$. 
\end{conjecture}
Aside from some low genus cases discussed below, Conjecture \ref{conj-can} is known only in two instances.
The first is that of the 2nd Hilbert point ($p=0$), which is established in two different ways in \cite{AFS-stability} and \cite{fedorchuk-jensen}. The second is that of the first linear syzygies among quadrics ($p=1$)  which is proved in \cite{DFS-syzygy}
for generic curves of odd genus. The proof in \cite{DFS-syzygy} proceeds by considering syzygies
of a special canonically embedded  non-reduced Gorenstein curve $R_{2k+1}$ of genus $g=2k+1$,
called a \emph{balanced canonical ribbon}. Considered earlier in \cite{AFS-stability}, 
$R_{2k+1}$ is the unique canonical ribbon in the sense of \cite{BE} 
with a $\GG_m$-action and maximal Clifford index. Explicitly, $R_{2k+1}$ is the section 
of the K3 carpet $S_{2k+1}$ (defined by Equation \eqref{E:carpet}) by the hyperplane $x_{k}-x_{x+1}=0$. 
Computer computations
for small values of $k$ give a hope that in fact the following more precise conjecture is true:
\begin{conjecture}\label{conj-ribbon}
The  balanced canonical ribbon $R_{2k+1}$ has a semistable syzygy point $\Sp_{(p,2)}(R_{2k+1})$ 
for all $2\leq p\leq k-1$.
\end{conjecture} 
We note that $K_{k-1,2}(R_{2k+1})=0$ by \cite{deopurkar-carpets}, and so all of the above syzygy points are well-defined.

\subsection{Hassett-Keel program for $\M_g$ and syzygies}
Assuming that Conjecture \ref{conj-can} is true, what could we possibly deduce from it? To answer 
this question using the framework of Section \ref{S:families}, whose notation we keep,
we take $T=\Mg{g}$ and $\pi\colon \cX \to \Mg{g}$ to be the universal curve. 
Then the line bundle $\cL:=\omega_{\pi}$ given by the relative dualizing sheaf
satisfies Assumption \ref{assumption-family}. In particular, $\pi_*{\cL}$ 
is the usual Hodge bundle over $\Mg{g}$, and so $c_1\bigl(\pi_*{\cL}\bigr)=\lambda$.
It follows that $c_1(\cQ)=-\frac{1}{g}\lambda$. Recall that by the Grothendieck-Riemann-Roch, we have
\[
\rank(\pi_*\omega^{n}_{\pi})=\begin{cases} g &  \text{for $n=1$}, \\ (2n-1)(g-1) & \text{for $n\geq 2$},
\end{cases}
\]
and 
\[
c_1(\pi_*\omega^{n}_{\pi})=\binom{n}{2}\kappa+\lambda,
\]
where $\kappa=\pi_*\bigl(c_1(\omega_{\pi})^2\bigr)=12\lambda-\delta$.
Using this, we compute that 
\[
c_1(\cE_n)=\binom{n}{2}\kappa+\lambda-n(2n-1)\frac{(g-1)}{g}\lambda.
\]
We therefore compute, using Lemma \ref{L:summation}, that 
\begin{align*}
c_1(\cS_{p,q}) &=\sum_{i=0}^p (-1)^i \binom{g}{p-i} \left(c_1\bigl(\cE_{q+i}\bigr)-
(q+i)(2(q+i)-1)(g-1)c_1(\cQ)\right) \\
&= \left(\binom{g-3}{p-2}-q\binom{g-2}{p-1}+\binom{q}{2}\binom{g-1}{p}\right)\left(\left(8+\frac{4}{g}\right)\lambda-\delta\right)\\ & \hspace{1pc} -\frac{(g-1)}{g}\left(q\binom{g-1}{p}-\binom{g-2}{p-1}\right)\lambda+\binom{g-1}{p}\lambda.
\end{align*}
We consider the two extreme
cases of $p=0$ and $q=2$. 
When $p=0$, we are of course in the realm of Hilbert points, and so $\cS_{0,q}$ is simply a (twisted) 
$q^{th}$ Hodge bundle over $\Mg{g}$.
We see that 
\[
c_1(\cS_{0,q})=\binom{q}{2}\left(\left(8+\frac{4}{g}\right)\lambda-\delta\right)
-\left(q-1-\frac{q}{g}\right)\lambda.
\]
When $q=2$, we have that $\cS_{p,2}$ is the bundle of $p^{th}$-order linear
syzygies among quadrics, and we calculate that 
\begin{equation}\label{E:syzygy-quadrics}
c_1(\cS_{p,2})=\binom{g-3}{p}\left[\left(8+\frac{4}{g}-\frac{(g-1)(g-2)}{g(g-p-1)}\right)\lambda-\delta\right].
\end{equation}

Divisors of this form arise naturally in  \emph{the Hassett-Keel program} for $\Mg{g}$, whose main goal 
is to construct modular interpretations of the following log canonical models 
of $\Mg{g}$:
\begin{multline}\label{E:log-canonical-1}
\M_{g}(\alpha):=\proj \oplus_{m\geq 0} \ \HH^0\bigl(\Mg{g},  m(K_{\Mg{g}}+\alpha \delta) \bigr) \\
=\proj \oplus_{m\geq 0} \ \HH^0\bigl(\Mg{g},  m(13\lambda-(2-\alpha) \delta) \bigr),
\ \text{ where } \alpha\in [0,1]\cap\QQ.
\end{multline}
This program 
was initiated by Hassett and Hyeon who discovered that by studying GIT of $k$-canonically embedded
curves for smaller values of $k$ than in the original GIT construction of $\M_g$ by Gieseker \cite{gieseker}, 
one obtains moduli stacks of more and more singular curves. These new GIT quotients 
miraculously turn out to be isomorphic to $\M_{g}(\alpha)$ for $\alpha > 7/10 -\epsilon$ \cite{HH1,HH2}. 

While the GIT approach was successful in constructing the first two steps in the Hassett-Keel
program for $\Mg{g}$, it quickly became clear that classical Hilbert stability 
constructions do not produce the next
step in the program. On the other hand, 
heuristic computations predict that already the next step in the Hassett-Keel program, if constructed
via GIT, necessitates a stability analysis of $6^{th}$ Hilbert points of bicanonical curves
\cite{morrison-git}. 
A divisor class computation of Equation \eqref{E:syzygy-quadrics}
similarly suggests the following (perhaps overly optimistic) conjecture:
\begin{conjecture}\label{C:hk} Let $\operatorname{Hilb}(\PP^{g-1})_{p,2}^{\s}$ be the locus inside the Hilbert 
scheme of genus $g$ canonically embedded curves consisting of those curves with a semistable $(p,2)$-syzygy point.
Then 
\[
\operatorname{Hilb}(\PP^{g-1})_{p,2}^{\s}\gitq \SL(g) \simeq \proj \bigoplus_{m\geq 0} \HH^0\left(\Mg{g},  m\left(\left(8+\frac{4}{g}-\frac{(g-1)(g-2)}{g(g-p-1)}\right)\lambda-\delta\right) \right).
\]
\end{conjecture}
It goes without saying that to prove this conjecture in any particular case, one would need to have a good 
understanding of  GIT stability of syzygy points of canonical curves.  Aside from some generic stability results
discussed above, our knowledge here is very limited. Whatever 
understanding we have, it does suffice to work out the first non-trivial variant
of Conjecture \ref{C:hk} for genus six canonical curves. This result is described in more detail in Section \ref{S:genus-6} below.

\begin{remark} A recent work of Aprodu, Bruno, and Sernesi shows that in genus
greater than $10$ and gonality greater than $3$, the $(1,2)$-syzygy point of a 
canonical curve determines the curve uniquely, unless the curve is bielliptic, in which 
case the $(1,2)$-syzygy point of the curve coincides with that of 
a cone over an elliptic curve \cite[Theorem 1]{aprodu-bruno-sernesi}. 
Thus, it is natural to expect that for 
$g\geq 11$, the GIT
quotient $\operatorname{Hilb}(\PP^{g-1})_{1,2}^{\s}\gitq \SL(g)$ will be a birational
model of $\cM_{g}$ in which the hyperelliptic, trigonal, and bielliptic loci are contracted 
(or flipped).
\end{remark}

\section{GIT for syzygies of canonical genus six curves}
\label{S:genus-6}

The first instance where the consideration of syzygy points 
leads to a genuinely new moduli space is the case of genus $6$ curves, which is
the smallest genus for which the $(1,2)$-syzygy point of a canonical curve 
is well-defined and non-trivial. What aids the GIT stability analysis here 
is the beautiful geometry of canonical genus $6$ curves, given 
by a well-known story, which we now recall. 
A smooth genus $6$ curve can be exactly one of the following: hyperelliptic, trigonal, bielliptic, a plane quintic,
or a quadric section of 
an anti-canonically embedded degree $5$ (possibly singular) del Pezzo in $\PP^5$.
A generic curve appears only in the last case, and only on a smooth del Pezzo.  Quadric sections
of singular del Pezzos form a divisor in the moduli space called the Gieseker-Petri divisor $D_{6,4}$
(we follow the taxonomy of \cite{farkas-rimanyi} for the Gieseker-Petri divisors; in particular $D_{6,4}$
is the divisor in $\mathcal{M}_6$ of curves with a base-point-free $g^{1}_4$ for which the Petri map
is not injective).
Since a smooth del Pezzo $\Sigma$ of degree $5$ is unique up to an isomorphism, and has a
group of automorphisms isomorphic to $S_5$, there 
is a distinguished birational model of $\M_6$ given by 
\begin{equation}\label{E:last-model}
X_6:=\PP \HH^0(\Sigma, -2K_\Sigma) / S_5. 
\end{equation}
This was the model used by Shepherd-Barron to prove rationality of $\M_6$ \cite{barron-6}. It has also reappeared recently in the context of the Hassett-Keel program of $\M_6$, as the ultimate non-trivial log canonical model of $\M_6$ \cite{Muller6}. In this section,
we reinterpret $X_6$ using GIT of $(0,2)$ and $(1,2)$-syzygy points
of canonical genus $6$ curves.  This allows us to also construct the penultimate
log canonical model of $\M_6$, and to realize the contraction of the Gieseker-Petri divisor $D_{6,4}$ 
as a VGIT two-ray game. 

Consider a smooth non-hyperelliptic curve $C$ of genus $6$.  
By Max Noether's theorem, the canonical embedding of $C$ is a projectively normal
degree $10$ curve in $\PP^5$. We set $V:=\HH^0(C, K_C)$, 
and identify $C$ with its canonical model in $\PP V^{\vee} \simeq \PP^5$.
According to Schreyer \cite{Schreyer1986}, 
there are exactly two possible graded Betti tables of $C$, depending
on the Clifford index of $C$.

\subsection{Clifford index $1$}
We have $\Cliff(C)=1$ if and only if $C$ has either $g^{1}_{3}$ (i.e., $C$ is trigonal) or $g^{2}_5$ (i.e., $C$ is a plane quintic). 
In this case the Betti table is:
 \[
 \begin{matrix} 
 1 &    &      & 	& \\ 
   & 6 &  8	 & 3 & \\ 
   & 3 &  8	  & 6	& \\ 
   &    &	  & 	& 1 \end{matrix}
   \]
Since $\dim K_{1,2}(C)=3$, the $(1,2)$-syzygy point of $C$ is not defined,
and so we need to analyze only Hilbert points of $C$.  
In fact, already the stability of
2nd Hilbert point detects finer aspects of the curve's geometry.
\begin{prop} Suppose $C$ is a canonically embedded smooth genus six
curve with $\Cliff(C)=1$. Then, with respect to the $\SL(V)$-action, 
the 2nd Hilbert point 
\[
\Sp_{(0,2)}(C)\in \Grass\bigl(6, \Sym^2 V\bigr)
\] 
\begin{enumerate}
\item is strictly semistable with $\dim \Stab=8$ if and only if $C$ is a plane quintic.
\item is strictly semistable with $\dim \Stab=6$ if and only if $C$ is trigonal with Maroni
invariant $0$.
\item is unstable if and only if $C$ is trigonal with positive Maroni invariant.
\end{enumerate}
\end{prop}

\begin{proof}
The key observation is that the quadrics containing $C$ cut out a 
surface $S$ of minimal degree in $\PP^5$ such that 
$\Sp_{(0,2)}(C)=\Sp_{(0,2)}(S)$. 
The stability analysis of $\Sp_{(0,2)}(S)$ 
is greatly simplified by the fact that $S$ is a rational 
surface with a large automorphism group. 

If $C$ is trigonal, then the canonical embedding of $C$ 
lies on a rational normal surface scroll $S_{a,4-a}$ in $\PP^5$, where $a\in \{1,2\}$
(see \cite[pp.12-13]{coppens} for a modern exposition of the classical 
work of Maroni \cite{maroni} and for a discussion of surface scrolls).
The homogeneous ideal of $S_{a,4-a}$ is 
generated by the $2 \times 2$ minors
of the following matrix
\[
 \left(\begin{array}{cccc|cccc}
x_{0} & x_{1} & \cdots & x_{a-1} & x_{a+1} & x_{a+2} & \cdots & x_{4} \\
 x_{1} & x_{2} & \cdots & x_{a} & x_{a+2} & x_{a+3} & \cdots & x_{5}
 \end{array}\right)
 \]
The Maroni invariant of $C$ is $\vert 4-2a\vert$, and equals to $0$ if and only
if the scroll is balanced. In the latter case, $S_{2,2} \simeq \PP^1\times \PP^1$,
embedded by the linear system $\vert\cO_{\PP^1\times \PP^1}(1,2)\vert$. Since
$\HH^0(\PP^1\times \PP^1, \cO(1,2))$ is an irreducible representation of $\SL(2)\times \SL(2)\subset
\Stab(S_{2,2})$,
we conclude that $\Sp_{(0,2)}(S_{2,2})=\Sp_{(0,2)}(C)$ is strictly semistable by Theorem \ref{T:kempf-luna},
and has $\dim \Stab=6$. Suppose now $C$ is trigonal with a positive Maroni invariant. Since $\Sp_{(0,2)}(C)=\Sp_{(0,2)}(S_{a,4-a})$, 
it suffices to destabilize $\Sp_{(0,2)}(S_{a,4-a})$, which is easily done using the one-parameter subgroup acting with weight
$a-5$ on $x_0, \dots, x_a$, and weight $a+1$ on $x_{a+1}, \dots, x_{5}$ (cf. \cite{fedorchuk-jensen}).

If $C$ is a plane quintic, then it lies on the Veronese surface $S=v_2(\PP^2)\subset \PP^5$. The homogeneous ideal of $v_2(\PP^2)$ is 
generated by the $2 \times 2$ minors
of the following symmetric matrix
\[
 \left(\begin{matrix}
x_{0} & x_{1} & x_2 \\ x_1 & x_3 & x_4 \\
 x_{2} & x_{4} &  x_5
 \end{matrix}\right)
 \]
 Since the Veronese is embedded into $\PP^5$ by an irreducible representation of $\SL(3)$,
we conclude that $\Sp_{(0,2)}(v_2(\PP^2))=\Sp_{(0,2)}(C)$ is strictly semistable by Theorem \ref{T:kempf-luna}, with $\dim \Stab=8$.

\end{proof}

\subsection{Clifford index $2$}
\label{S:cliff-2}
In this case the Betti table of $C$ is:
 \[
 \begin{matrix} 
 1 &    &      & 	& \\ 
   & 6 &  5	  &    & \\ 
   &    &  5	  & 6	& \\ 
   &    &	  & 	& 1 \end{matrix}
   \]
We see that the canonical embedding of $C$ satisfies $K_{1,2}(C)=0$,
and so has both a well-defined 
syzygy point $\Sp_{(0,2)}(C)\in \Grass(6, \Sym^2 V)$ 
and a well-defined 
syzygy point $\Sp_{(1,2)}(C)\in \Grass(5, \mathbb S^{(2,1)}(V))$. Given such a curve $C$, we can thus associate to it a point 
\begin{equation}\label{E:h}
h(C):=(\Sp_{(0,2)}(C), \Sp_{(1,2)}(C)) \in \mathbb H:=\Grass(6, \Sym^2 V) \times  \Grass(5, \mathbb S^{(2,1)}(V)).
\end{equation}
The $\SL(V)$-linearized semiample cone of $\mathbb H$ is spanned by the pullbacks  $p_i^*\cO(1)$ of the Pl\"ucker line bundles from the two factors. This gives rise to a two-ray VGIT game, whose endpoints correspond to GIT for $(0,2)$ and $(1,2)$-syzygy points, respectively. 

The $(1,2)$-syzygy point of $C$
coincides with the $(1,2)$-syzygy point of the unique degree $5$ surface $S$
containing $C$, called \emph{the second syzygy scheme}
of $C$; see \cite{aprodu-bruno-sernesi}. There are two different possibilities:
\begin{enumerate}
\item $S$ is a degree $5$ del Pezzo surface, namely, a blow-up of $\PP^2$
at four possibly infinitely near points, 
embedded anti-canonically into $\PP^5$. 
\item $S$ is a cone over an elliptic normal curve of degree $5$ in $\PP^4$.
\end{enumerate}
In both cases $h^0(\PP^5,\cI_S(2))=5$, and $C$ is a quadric section of $S$. 
This shows that the space of quadrics cutting out $C$ is a span of five quadrics
cutting out $S$ and another free quadric:
\[
\HH^0(\PP^5,\cI_C(2))=\HH^0(\PP^5,\cI_S(2))+\CC\langle Q\rangle \subset \Sym^2 V,
\]
and the only linear syzygies among 
the quadrics are those coming from $S$.

\subsubsection{Bielliptic curves and elliptic ribbons} 
\label{S:bielliptic}
In genus $6$, bielliptic curves arise as quadric sections of a cone
over a genus one smooth curve $E$ embedded into $\PP^4$ 
by a complete linear system $\vert \cL\vert$ of a line bundle $\cL\in \Pic^5(E)$. 
A related construction that also produces a genus $6$ curve out of the same datum is
given by ribbons of Bayer and Eisenbud \cite{BE}. Given an elliptic curve $E$ and $\cL\in \Pic^5(E)$,
we can consider an infinitesimal thickening $\cO_{R}$ of $\cO_E$ by $\cL$. More precisely, we consider 
a short exact sequence of sheaves on $E$ of the form
\[
0 \to \cI \to \cO_R \to \cO_E \to 0,
\]
where $\cO_R$ is a sheaf of $\CC$-algebras, and $\cI$ is a sheaf of principal ideals such that $\cI^2=0$
and $\cI \simeq \cL^{-1}$ as an $\cO_E$-module. We denote by $R$ the non-reduced curve with structure sheaf 
$\cO_R$ supported on $E$, and call $R$ an \emph{elliptic ribbon}. 
By construction, $R$ can be specified as an element of $\operatorname{Ext}^{1}_{\cO_E}(\cO_E, \cL^{-1})\simeq \HH^0(E, \cL)^{\vee}$. Via the natural $\GG_m$-action given by scaling the extension class, 
every genus $6$ 
elliptic ribbon on $E$ can be isotrivially degenerated to the split ribbon
\begin{equation}\label{E:split-ribbon}
R_{E}^{\cL}:=\mathit{Spec}_{\cO_E} \left(\left(\cO_E\oplus \epsilon \cL^{-1}\right)/(\epsilon^2)\right).
\end{equation}
The connection between elliptic ribbons and bielliptic curves is the following. If $E$ is embedded into $\PP^4$
by the complete linear system $\vert \cL\vert$ and $\Cone(E)$ is a cone over it in $\PP^5$, then $R_{E}^{\cL}$
is isomorphic to the double hyperplane section of $\Cone(E)$. In particular, $R_{E}^{\cL}$ is an
isotrivial specialization of every smooth
bielliptic genus $6$ quadric section of $\Cone(E)$. 

We are now ready to state the following result:

\begin{prop}\label{P:bielliptic} 
\hfill

\begin{enumerate}
\item[(a)] 
Suppose $C$ is a smooth bielliptic curve. 
Then $\Sp_{(0,2)}(C)\in \Grass(6,\Sym^2 V)$ is strictly 
semistable, $\Sp_{(1,2)}(C)\in \Grass(5,\mathbb S^{(2,1)} V)$ is unstable, and $h(C)$
is unstable with respect to any ample linearization on $\Grass(6,\Sym^2 V)\times \Grass(5, \mathbb S^{(2,1)}(V))$.
\item[(b)] 
Suppose $C$ is a genus $6$ elliptic ribbon in $\PP^5$. Then $\Sp_{(0,2)}(C)\in \Grass(6,\Sym^2 V)$ is strictly 
semistable.
\end{enumerate}
\end{prop} 
\begin{proof} Suppose $C$ is a smooth bielliptic curve. As we have discussed,
$C$ is a quadric section of a cone $S=\Cone(E)$ over a normal elliptic curve $E\subset \PP^4$,
where $E$ is embedded by a complete linear system $\vert \cL\vert$ of degree $5$.
Choose a basis $x_0, x_1,\dots, x_5$ of $\HH^0(C,K_C)^{\vee}$ such that the vertex of
the cone $S$ is given by $x_1=\cdots=x_5=0$. Let
$\rho$ be the one-parameter subgroup of $\SL(V)$ acting with weight $-5$ on $x_0$
and weight $1$ on each of $x_1, \dots, x_5$.  Then all five quadrics in $\HH^0(\PP^5,\cI_S(2))$ are homogeneous of weight $2$, while the smallest weight
term of the free quadric $Q$ has weight $-10$. At the same time, 
all six syzygies in $\Sp_{(1,2)}(C)=\Sp_{(1,2)}(S)$ are homogeneous of weight $3$. 
This proves all instability claims.  

It remains to show that $\Sp_{(0,2)}(C)$ is actually semistable. 
Note that \[
\lim_{t\to 0} \rho(t)\cdot \Sp_{(0,2)}(C)=\Sp_{(0,2)}(R),
\]
where $R=R_{E}^{\cL}$ is the canonically embedded 
genus $6$ split ribbon supported on $E$ and given by Equation \eqref{E:split-ribbon}.
Thus it suffices to prove that $R$ is semistable.

Explicitly, $I_R=(I_S, x_0^2)$ and the one-parameter subgroup $\rho$ introduced above is a subgroup of $\Stab(R)$.
By Theorem \ref{T:kempf-luna},
it suffices to verify semistability of $\Sp_{(0,2)}(R)$ with respect to the 
centralizer of $\rho$ in $\SL(6)$. Let $\mu$ be a 
one-parameter subgroup of this centralizer. Suppose $\mu$ acts on $x_0$
by $x_0 \mapsto t^{w_0}x_0$. Since $\Sp_{(0,2)}(R)$ is strictly semistable
with respect to $\rho$, it will be $\mu$-semistable if and only if
it is semistable with respect to the one-parameter subgroup $\mu':=\rho^{-w_0} \mu^{5}$.
Since $\mu'$ acts trivially on $x_0$ and acts via a one-parameter subgroup of $\SL(5)$ on $\CC\langle x_1,\dots, x_5\rangle$, 
we have reduced to verifying $\mu'$-semistability of the 2nd Hilbert point 
of $E\subset \PP^4$.   However, $E$ is an abelian variety, embedded by a complete linear
series, and so has a semistable 2nd Hilbert point by \cite[Corollary 5.2]{kempf}.  

Having established semistability of all canonical split elliptic ribbons of genus $6$, 
we immediately obtain Part (b), as every elliptic ribbon isotrivially
degenerates to a split ribbon.
\end{proof}

\subsubsection{Singular degree $5$ del Pezzo}
\label{S:singular-del-pezzo} 
Consider now a degree $5$ del Pezzo surface $\Sigma_0$ with exactly one ordinary 
double point and no other singularities, which is constructed as follows. 
Let $x, y, z$ be the standard coordinates on $\PP^2$, and let $X$ be the blow-up of 
$\PP^2$ at the points $p_1:=[1:0:0]$, $p_2:=[1:1:0]$, $p_3:=[0:1:0]$ and
$p_4:=[0:0:1]$. The first three of these points lie on the line $z=0$, and
so the strict transform of this line in $X$ is a $(-2)$-curve. The anti-canonical
line bundle of $X$ is globally generated, and defines a morphism to $\PP^5$
whose image is  $\Sigma_0$, with the $A_1$-singularity of $\Sigma_0$ being precisely 
the contraction of the  $(-2)$-curve on $X$. 

The automorphism group of $\Sigma_0$ 
will play an important role in our GIT stability analysis, so we note first 
that $\Sigma_0$ admits a $\GG_m$-action, induced
by the scaling action on $\PP^2$ that fixes $p_4$ and the line $z=0$. We
then note that there is also an action of the symmetric group $S_3$, induced by 
the action of $S_3$ on $\PP^2$ permuting
the three points $p_1, p_2, p_3$, and leaving $p_4$ fixed.   

\begin{prop}\label{P:singular-del-pezzo} Let $C$ be a quadric section 
of $\Sigma_0$ that does not pass through the singular point of $\Sigma_0$. 
Consider an ample linearization of the $\SL(V)$-action on 
$\mathbb H$ given by $\cO(1) \boxtimes \cO(\beta)$.
Then $h(C)=(\Sp_{(0,2)}(C), \Sp_{(1,2)}(C))$ is: 
\begin{enumerate}
\item unstable if $\beta>4$.
\item strictly semistable if $\beta=4$. 
\end{enumerate}
\end{prop}

\begin{proof}
The natural scaling action of $\GG_m$ on $\PP^2$, given by 
$t\cdot [x:y:z]=[t^{-1}x:t^{-1}y:z]$, extends to $X$, and 
gives rise to a one-parameter subgroup $\rho$ of $\Stab(\Sigma_0) \subset \SL(V)$. To understand
this subgroup, 
we begin by choosing a $\GG_m$-semi-invariant basis 
of $\HH^0(X, -K_X)$: 
\begin{equation}\label{E:S_0-basis}
a:=xy(x-y),  b_1:=zx^2, b_2:=zxy, b_3:=zy^2, c_1:=z^2x, c_2:=z^2y.
\end{equation}
Re-normalizing the weights of the action so as to obtain a one-parameter subgroup of $\SL(6)$, we see that $\rho$ acts via
\[
t\cdot (a, b_1, b_2, b_3, c_1, c_2)=(t^{-7} a, t^{-1}b_1, t^{-1}b_2, t^{-1}b_3,  t^{5}c_1, t^{5}c_2).
\]
Recall that $S_3$ acts on $\Sigma_0$ and 
note that $U_1:=\langle c_1,c_2\rangle$ is the standard representation 
of $S_3$, while $U_2:=\langle b_1,b_2, b_3\rangle$ is its second symmetric
power. It follows that $\HH^0(\Sigma_0, \cO_{\Sigma_0}(1))\simeq \HH^0(X, -K_X)$ is a multiplicity-free representation
of $\Aut(\Sigma_0)$. 

The quadrics cutting out $\Sigma_0$ are given by the Pfaffians of the following anti-symmetric matrix:
\begin{equation}\label{E:pfaffian-sing}
\left(\begin{matrix} 
0 & z_1 & z_1 & z_2 & -z_0 \\
-z_1 & 0 & z_4 & 0 & z_2 \\
-z_1 & -z_4 & 0 & z_5 & z_3 \\
-z_2 & 0 & -z_5 & 0 & z_3 \\
z_0 & -z_2 & -z_3 & -z_3 & 0 
\end{matrix}\right)
\end{equation}
A routine computation now shows that $\rho$ acts on $\det \HH^0\bigl(\PP^{5}, \cI_{\Sigma_0}(2)\bigr)$
with weight $2$, and on $\det \HH^0\bigl(\PP^{5}, \cI_{\Sigma_0}(3)\bigr)$ with weight $9$. 
We also record that $\HH^0(\Sigma_0,\cO_{\Sigma_0}(2))$ is a homomorphic image of
\begin{equation}\label{E:singular-quadrics}
\langle a^2\rangle \oplus a U_2 \oplus \Sym^2 U_1 \oplus
\langle b_1^2, b_1b_2, b_2^2, b_2b_3, b_3^2\rangle \oplus
\langle b_1c_1,b_2c_1, b_2c_2, b_3c_2\rangle
\end{equation}
which also shows that $\rho$ acts on $\HH^0(\Sigma_0,\cO_{\Sigma_0}(2))$
with weights \[(-14, -8, -8, -8, 10, 10, 10, -2, -2, -2, -2, -2, 4, 4, 4, 4),\] and on
$\det \HH^0(\Sigma_0,\cO_{\Sigma_0}(2))$ with weight 
$-2$. From this, we conclude that the $\rho$-weight of 
$\Sp_{(1,2)}(\Sigma_0)$ is $-3$, 
which shows in particular that $\Sp_{(1,2)}(\Sigma_0)$ is unstable.

Suppose now $C\subset \Sigma_0$ is a quadric section of $\Sigma_0$, given 
by an equation $Q(z_0,\dots,z_6)=0$. 
Since $\HH^0(C,\cI_C(2))=\HH^0(\Sigma_0,\cI_{\Sigma_0}(2))\bigoplus \CC\langle Q\rangle$,
we see that $\rho$ acts on $\Sp_{(0,2)}(C)$ with total 
weight 
\[
w_{\rho}(\Sp_{(0,2)}(C))= -w_\rho(Q)+w_{\rho}(\Sp_{(0,2)}(\Sigma_0)) \leq 14-2=12,
\]
with equality holding if and only if the lowest $\rho$-weight term of $Q$ is $z_0^2$ if and only 
if $C$ does not pass through the singularity of $\Sigma_0$. 

Continue with the assumption that $C$ does not pass through the singularity of $\Sigma_0$. Set
\[
C_0:=\lim_{t\to 0} \rho(t)\cdot C,
\]
where we are taking a flat limit in $\PP^5$. Then it is easy to see that 
\begin{multline}\label{E:gp-curve}
I_{C_0}=(z_0^2, I_{\Sigma_0})\\ =(z_0^2, 
z_3z_4-z_2z_5,
z_2z_4-z_1z_5,
z_2^2-z_1z_3,
z_1z_3-z_2z_3-z_0z_5,
z_1z_2-z_0z_4+z_1z_3).
\end{multline}
The curve $C_0$ will play a prominent role in what follows. We note some of its properties.
To begin, $C_0$ is a canonically embedded Gorenstein curve. Scheme-theoretically, it is a union of 
a fourfold line $L_0=\VV(z_4-z_5,z_2-z_3,z_1-z_3,z_0)$ and three double lines 
$L_1=\VV(z_4,z_2,z_1,z_0)$, $L_2=\VV(z_3,z_2,z_1,z_0)$, and $L_3=\VV(z_5,z_3,z_2,z_0)$.
By construction, $C_0$ is also the unique $\GG_m$-invariant double hyperplane section of $\Sigma_0$
that does not pass through the singular point of $\Sigma_0$. 
The statement of the proposition now follows from the following lemma:

\begin{lemma}\label{L:C0} Consider the linearization 
$\cO(1)\boxtimes \cO(\beta)$ on $\mathbb H$. Then $C_0$ is polystable for $\beta=4$.
Furthermore, $C_0$ is destablized by $\rho$ when
$\beta>4$, and is destabilized by $\rho^{-1}$ when $\beta<4$. 
\end{lemma}
\begin{proof}
By the above computation, we have 
\begin{align*}
w_{\rho}(\Sp_{(0,2)}(C_0))&= w_\rho(z_0^2)+w_{\rho}(\Sp_{(0,2)}(\Sigma_0)) = 14-2=12, \\
w_{\rho}(\Sp_{(1,2)}(C_0)) &= w_{\rho}(\Sp_{(1,2)}(\Sigma_0)) = -3. 
\end{align*}
The instability claims follow. 

It remains to show that $C_0$ is polystable for $\beta=4$. For this we note that
$C_0$ is fixed by $\Aut(\Sigma_0)$, and so in particular $\Aut(C_0)=\Aut(\Sigma_0)\simeq \GG_m\times S_3$.
Hence, by Theorem \ref{T:kempf-luna}, it suffices to verify polystability of $C_0$
with respect to those one-parameter subgroups of $\SL(V)$ that commute with $\GG_m\times S_3 \subset \Aut(C_0)$.
Any such one-parameter subgroup acts diagonally on the distinguished basis
of $V$ given by Equation \eqref{E:S_0-basis} as follows 
\[
t\cdot (z_0,\dots, z_5)=(t^{w_0} z_0, t^{w_1} z_1, t^{w_1} z_2, t^{w_1} z_3, t^{w_2} z_4, t^{w_2} z_5),
\]
where $w_0+3{w_1}+2{w_2}=0$. Since the $\rho$-action stabilizes $C_0$, it suffices to check
the polystability of $C_0$ only with respect to the one-parameter subgroup given by weights $(-2, 0, 0, 0, 1, 1)$
and $(2, 0, 0, 0, -1, -1)$.
This is a routine calculation using the fact that the quadrics and syzygies of $\Sigma_0$
are those satisfied by the Pfaffians of the matrix \eqref{E:pfaffian-sing}.
\end{proof}
\noqed
\end{proof}

\subsection{The Hassett-Keel program for $\M_6$}

In this subsection, we bootstrap the stability results of \S\ref{S:cliff-2} in order
to show that the last two steps in the Hassett-Keel program for $\M_6$ are constructed
via GIT for syzygy points of canonically embedded curves. 
 
Let $V=\CC^6$.
Recall that $\mathbb H=\Grass(6, \Sym^2 V) \times  \Grass(5, \mathbb S^{(2,1)}(V))$
is equipped with a two-dimensional $\SL(V)$-ample cone, whose elements we denote by $\cO(1) \boxtimes \cO(\beta)$.
For every smooth canonically embedded curve $C\subset \PP V^{\vee}$ of genus six and Clifford
index $2$, we have a well-defined point $h(C)=(\Sp_{(0,2)}(C), \Sp_{(1,2)}(C))\in \mathbb H$.
Denote by $\cQ \subset \mathbb H$ the Zariski closure of the locus of all such $h(C)$ in $\mathbb H$.

We have a natural $\SL(V)$-action on $\cQ$, and we denote by $\cQ^{\s}(\beta)$ the semistable locus in $\cQ$
with respect to the linearization $\cO(1) \boxtimes \cO(\beta)$. We also let 
$\cG(\beta):=[\cQ^{\s}(\beta)/\SL(V)]$ be the corresponding GIT quotient stack and denote by
$\overline G(\beta)$ its moduli space. Namely, we have that 
\[
\overline G(\beta)=\cQ^{\s}(\beta)\gitq \SL(V).
\]

We recall that $\Sigma$ is a smooth degree $5$ del Pezzo and $\Sigma_0$ is a degree $5$ del
Pezzo with a unique $A_1$-singularity. Let $\cU\subset \PP\HH^0(\Sigma_0, -2K_{\Sigma_0})$ be
the open locus of quadric sections of $\Sigma_0$ not passing through the singularity. 
We recall that $C_0$, given by Equation \eqref{E:gp-curve}, is the unique curve with $\GG_m$-action 
in $\cU$. We note that, on a smooth del Pezzo, there exists a unique $C'_0 \in \PP\HH^0(\Sigma, -2K_{\Sigma})$
such that $C'_0$ isotrivially specializes to $C_0$. Explicitly, we have
\[
C'_0=4F_1+2F_2+2F_3+2F_4,
\] 
where $F_i$'s are $(-1)$-curves on $\Sigma$ such that
$F_2, F_3, F_4$ meet $F_1$, and have no other pairwise intersections; $C'_0$ is the same curve
as described in \cite[Proposition 2.6]{Muller6}.

We will also need the following counterpart of Proposition \ref{P:singular-del-pezzo}:
\begin{lemma}\label{L:smooth-del-pezzo}
For every $C\in \PP\HH^0(\Sigma, -2K_{\Sigma})$, the syzygy point $\Sp_{(1,2)}(C)$ 
is semistable. 
\end{lemma}
\begin{proof} The second syzygy scheme of $C$ is $\Sigma$, and so $\Sp_{(1,2)}(C)=\Sp_{(1,2)}(\Sigma)$.
Since $\Sigma$ is embedded into $\PP^5$ by an irreducible representation of $\Aut(\Sigma)=S_5$ by
\cite{barron-6}, the semistability of $\Sp_{(1,2)}(\Sigma)$ follows from Theorem \ref{T:kempf-luna}.
\end{proof}
Having fixed all the notation, we are ready to state the main result of this section:

\begin{theorem}[Contraction of the Gieseker-Petri divisor in $\M_6$ via VGIT]
\begin{enumerate}
\item[]
\item For $\beta>4$, $\cQ^{\s}(\beta)=\{h(C) \mid C\in \PP\HH^0(\Sigma, -2K_{\Sigma})\}$.
Moreover, 
\[
\cG(\beta)\simeq [\PP \HH^0(\Sigma, -2K_{\Sigma})/S_5]
\]
is a Deligne-Mumford stack and $\overline G(\beta) \simeq X_6 \simeq \M_6(\alpha)$, where $\alpha\in (16/47, 35/102)$.

\item For $\beta=4$, $\cQ^{\s}(4)=\cQ^{\s}(4+\epsilon) \cup \{h(C) \mid C\in \cU\}$. Moreover, 
\[
\overline G(4)\simeq \M_6(35/102).
\]
\item For $\beta\in (4-\epsilon, 4)$, 
\[
\cQ^{\s}(\beta)=\{h(C) \mid C\in \cU, \ C\neq C_0\}\cup 
\{h(C) \mid C\in \PP\HH^0(\Sigma, -2K_{\Sigma}), \ C\neq C'_0\}.\]
The stack $\cG(\beta)$ is Deligne-Mumford, and we have 
\[
\overline G(\beta)\simeq \M_6\left(\frac{35}{102}+\epsilon\right).
\]
\item We have a commutative diagram
\begin{equation}\label{E:VGIT}
\begin{aligned}
\xymatrix{
\cG(4-\epsilon) \ar@{^(->}[r] \ar[d]	
& \cG(4) \ar[dd] & \cG(4+\epsilon) \  \ar[dd]  \ar@{_(->}[l]  \\
\M_{6}(\frac{35}{102}+\epsilon)\ar[rd]
& 			
& 	\\
&	   \M_{6}(\frac{35}{102})     & \M_{6}(\frac{35}{102}-\epsilon) \ar[l]_{\sim} \\
}
\end{aligned}
\end{equation}
\end{enumerate}

Moreover, 
$\M_6(\frac{35}{102}+\epsilon)$ is isomorphic to $\M_6$ at the generic point of the Gieseker-Petri
divisor $D_{6,4}$ and $\M_{6}(\frac{35}{102}+\epsilon) \to \M_{6}(\frac{35}{102})$ is the contraction 
of this divisor to $[C_0]\in \M_{6}(\frac{35}{102})$. 

\end{theorem}

\begin{proof} A starting point for us is the fact that the generic point of $\cQ$ lies on a smooth del Pezzo of degree $5$. Thus the projection of $\cQ$ to $\Grass(5, \mathbb S^{(2,1)}(V))$ consists of the closure of a single orbit 
$\SL(V)\cdot \Sp_{(1,2)}(\Sigma)$. We have seen in Propositions \ref{P:bielliptic} 
and \ref{P:singular-del-pezzo} that for $\beta>4$ any point of $\cQ$ lying over the boundary 
of this orbit's closure is unstable, and for $\beta=4$ only the point lying over 
$\SL(V)\cdot \Sp_{(1,2)}(\Sigma_0)$ becomes strictly semistable. 
It follows that $\beta=4$
is the first wall, as $\beta$ decreases, 
where the semistability locus changes, and so for $\beta>4-\epsilon$,
a point in $\cQ$ is semistable only if its projection to $\Grass(5, \mathbb S^{(2,1)}(V))$ 
lies in the union of the two orbits
$\SL(V)\cdot \Sp_{(1,2)}(\Sigma) \cup \SL(V)\cdot \Sp_{(1,2)}(\Sigma_0)$.

(1) For $\beta>4$, we have that $\cQ^{\s}(\beta) \subset \{h(C) \mid C\in \PP\HH^0(\Sigma, -2K_{\Sigma})\}$. 
The non-emptiness of $\cQ^{\s}(\beta)$ follows from the semistability of $\Sp_{(1,2)}(C)$ given by Lemma 
\ref{L:smooth-del-pezzo}
and the generic semistability of $\Sp_{(0,2)}(C)$. Since $\Sigma$, and hence its every quadric section, 
has no infinitesimal automorphisms, every semistable point is stable. Using the properness of the GIT quotient stack $\cG(\beta)$, we conclude
that in fact we must have an equality
\[
\cQ^{\s}(\beta) = \{h(C) \mid C\in \PP\HH^0(\Sigma, -2K_{\Sigma})\}.
\]
It follows that 
\[
\cG(\beta)\simeq \{h(C) \mid C\in \PP\HH^0(\Sigma, -2K_{\Sigma})\}/\SL(V) \simeq \PP\HH^0(\Sigma, -2K_{\Sigma})/\Aut(\Sigma)=X_6.
\]

(2) Since $\beta=4$ is the first wall (as $\beta$ decreases) where the semistability changes,
we have $\{h(C) \mid C\in \PP\HH^0(\Sigma, -2K_{\Sigma})\}\subset \cQ^{\s}(4)$. The inclusion 
$\{h(C) \mid C\in \cU\}\subset \cQ^{\s}(4)$ follows by Proposition \ref{P:singular-del-pezzo}. 
As we have already discussed, there can
be no other semistable points for $\beta=4$. 

Next we note that $C_0$ is the only point in 
$\cG(4)\setminus \left(\cG(4-\epsilon)\cup \cG(4+\epsilon)\right)$. The basin of attraction of
$C_0$  (that is the locus of points isotrivially degenerating to $C_0$) in $\cG(4-\epsilon)$ consists
of all of $\cU$, and the basin of attraction of
$C_0$  in $\cG(4+\epsilon)$ consists of a single point, $C'_0 \in \PP\HH^0(\Sigma, -2K_{\Sigma})$. 
It follows (by normality of the GIT quotients) 
that the open immersion of stacks $\cG(4+\epsilon) \subset \cG(4)$ induces an isomorphism on the level
of their moduli spaces. Namely, $\overline G(4+\epsilon)\simeq \overline G(4)$. 

(3) Let $C$ be a \emph{generic} quadric section 
of $\Sigma_0$ that does not pass through the singular point of $\Sigma_0$. 
We are going to prove that $h(C)=(\Sp_{(0,2)}(C), \Sp_{(1,2)}(C))$ is stable for $0<\beta<4$.
To this end, consider a double smooth cubic passing through the points 
$p_1=[1:0:0]$, $p_2=[1:1:0]$, $p_3=[0:1:0]$ and
$p_4=[0:0:1]$ in $\PP^2$. Its strict transform in $\Sigma_0$ is an elliptic ribbon $R$ of genus $6$
that does not pass through the singular point of $\Sigma_0$. 
By Proposition \ref{P:singular-del-pezzo},
$h(R)=(\Sp_{(0,2)}(R), \Sp_{(1,2)}(R))$ is semistable with respect to the linearization $\cO(1) \boxtimes \cO(4)$.
By Proposition \ref{P:bielliptic}, $\Sp_{(0,2)}(R)$ is strictly semistable. 
It follows that $h(R)$ is semistable 
with respect to all linearizations $\cO(1) \boxtimes \cO(\beta)$, where $\beta\in (0,4)$.

To prove that $\cQ^{\s}(\beta)$ has no strictly semistable points, we note that for $\beta\in (4-\epsilon, 4)$ 
no semistable point admits a $\GG_m$-action. Indeed,
the smooth del Pezzo $\Sigma$ has no quadric sections with $\GG_m$-action, and the only quadric
section of the singular del Pezzo $\Sigma_0$ is $C_0$, which is unstable for $\beta<4$.
This proves that $\cG(\beta)$ is Deligne-Mumford for $\beta\in (4-\epsilon, 4)$.

It remains to show that \emph{every} quadric section $C$
of $\Sigma_0$ that does not pass through the singular point of $\Sigma_0$ and such that $C\neq C_0$ 
is in fact semistable for $\beta\in (4-\epsilon, 4)$. 
This follows from the properness of the GIT quotient stack. Indeed, the open 
inclusion $\{h(C) \mid C\in \cU\cap \cQ^{\s}(\beta)\} \subset \{h(C) \mid C\in \cU, \ C\neq C_0\}$ must induce a 
birational morphism of projective quotients 
\[
\{h(C) \mid C\in \cU\cap \cQ^{\s}(\beta)\}\gitq \SL(V) \to \left(\cU\setminus [C_0]\right)\gitq \Aut(\Sigma_0) \simeq 
\PP(6^3, 12^5, 24^3, 28^4)/S_3,
\]
where the last identification follows from the explicit description of the $\GG_m$-action on 
$\cU \subset \PP\HH^0(\Sigma_0, -2K_{\Sigma_0})$ given by Equation \eqref{E:singular-quadrics}.
Since $\cQ^{\s}(\beta)$ has no strictly semistable points, this morphism must be an isomorphism and 
so every element of $\cU\setminus [C_0]$ is stable in $\cQ^{\s}(4-\epsilon)$. In particular,
$\cG(4-\epsilon)$ is isomorphic to $\Mg{6}$ at the generic point of the Gieseker-Petri divisor $D_{6,4}$.

(4) Given our stability results, the existence of the commutative diagram \eqref{E:VGIT} follows from the general theory of VGIT. We finally address the identifications of the GIT quotients $\overline G(\beta)$ 
with log canonical models appearing in the Hassett-Keel program for $\M_6$. Since for all $\beta>4-\epsilon$, 
the stacks $\cG(\beta)$ parameterize
Gorenstein curves with ample dualizing sheaf, and the locus of worse-than-nodal curves in $\cG(\beta)$
has codimension at least $2$, 
it follows from Equation \eqref{E:syzygy-quadrics} (for $g=6$ and $p=0,1$) that the Pl\"ucker line 
bundles $p_1^*\cO_{\Grass(6,\Sym^2 V)}(1)$
and $p_2^*\cO_{\Grass(5, \mathbb S^{(2,1)}(V))}(1)$
descend to $8\lambda - \delta$ and $\frac{47}{2}\lambda-3\delta$, respectively, on the GIT quotient.
It follows that $\cO(1)\boxtimes \cO(\beta)$ descends to $\overline G(\beta)$ as
an ample line bundle
\begin{equation}\label{E:polarization}
(8\lambda - \delta)+\beta\left(\frac{47}{2}\lambda-3\delta\right).
\end{equation}
The isomorphism $\overline G(4\mp \epsilon) \simeq \M_6(35/102\pm \epsilon)$ now follows
from routine discrepancy computations as in \cite[Proposition 4.3]{Muller6},
where the identification of $X_6$ with the log canonical models $\M_6(\alpha)$ for $\alpha\in (16/47, 35/102)$
was first established.

\end{proof}

\begin{corollary} The moving slope of $\M_6$ is $102/13$.
\end{corollary}
\begin{proof} By \cite[Prop. 4.1]{Muller6}, the moving slope of $\M_6$ is at most $102/13$. It remains
to find a family $T$ of curves in $\M_6$ passing through the generic point of the Gieseker-Petri divisor $D_{6,4}$, 
such that $(T\cdot \delta_0)/(T\cdot \lambda)=102/13$, and such that $T$ avoids the boundary divisors $\delta_i$,
$i\geq 1$. The existence of such a family is immediate from the fact that the strict transform of $D_{6,4}$
in $\M_{6}(35/102+\epsilon)$ is the GIT quotient
\[
\mathfrak D:=(\cU\setminus [C_0]) \gitq \operatorname{Aut}(\Sigma_0) \simeq \PP(6^3, 12^5, 24^3, 28^4)/S_3.
\]
Indeed, since $\cU$ is an open subset of a complete linear system on $\Sigma_0$, it follows
that $\mathfrak D$ parameterizes at worst nodal irreducible curves away from codimension $2$. 
Since $\mathfrak D$ has Picard number $1$, all curves in $\mathfrak D$ 
have the same slope. The
existence of a requisite family $T$ follows from the fact that the line bundle 
\[
\cO(1)\boxtimes \cO(4)=\left(8\lambda-\delta\right)+4\left(\frac{47}{2}\lambda-3\delta\right)=102\lambda-13\delta
\]
is trivial on $\mathfrak D$.
\end{proof}

\subsubsection{Future directions} It is conceivable that one might be able
to complete the VGIT analysis of $\cQ^{\s}(\beta)$ for \emph{all} $\beta\in (0,4) \cap \QQ$. Since 
genus $6$ curves of Clifford index $0$ and $1$ have no well-defined $(1,2)$-syzyzy points, and bielliptic 
curves are unstable for all positive $\beta$ by Proposition \ref{P:bielliptic}, we expect that the
resulting GIT quotients 
\[
\overline G(\beta)=\cQ^{\s}(\beta)\gitq \SL(V)
\]
will parameterize canonical genus $6$ curves lying on arbitrary del Pezzo surfaces 
(in the sense of \cite[Definition 8.1.12]{CAG}) of degree $5$ in $\PP^5$. 
By \cite[\S8.5.1 and Table 8.5]{CAG}, there are just $6$ types of singular del Pezzos of degree $5$, 
corresponding to different root sublattices in a root lattice of type $A_4$,
with $\Sigma_0$ being the least singular of them.  
We expect that as $\beta$ decreases, $\overline G(\beta)$ will contain quadric sections of del Pezzos with worse and worse 
singularities, until for some small $\beta$, we will see all possible del Pezzo surfaces. 

If for all positive $\beta$ the GIT quotients $\overline G(\beta)$ do indeed parameterize only quadric sections of all 
degree $5$ del Pezzos, then, by a standard double 
covering construction, we will 
obtain a sequence of compact moduli spaces of K3 surfaces 
that fits into the Hassett-Keel-Looijenga program (see \cite{laza-ogrady})  
for the moduli space of K3 surfaces studied in \cite{artebani-kondo}.

\bibliographystyle{alpha}

\bibliography{NSF_bib}

\end{document}